\documentclass[a4paper,12pt,reqno]{amsart}
\usepackage{a4wide}
\usepackage{amsmath}
\usepackage{amssymb}
\usepackage{amsthm}
\usepackage{latexsym}
\usepackage{graphicx}
\usepackage[english]{babel}

\newtheorem{obs} [subsection]{Remark}
\newtheorem{exm} [subsection]{Example}

\newtheorem{prop}[subsection]{Proposition}

\newtheorem{teor}[subsection]{Theorem}
\newtheorem*{teor*}{Theorem}
\newtheorem{lema}[subsection]{Lemma}
\newtheorem{cor} [subsection]{Corollary}

\newcommand{\Zng}{$\mathbb Z^n$-graded $S$-module}

\def\supp{\operatorname{supp}}

\def\gcd{\operatorname{gcd}}
\def\Ass{\operatorname{Ass}}

\def\Mon{\operatorname{Mon}}

\def\depth{\operatorname{depth}}
\def\sdepth{\operatorname{sdepth}}

\def\slength{\operatorname{slength}}

\usepackage{mathtools}

\begin{document}
\selectlanguage{english}
\frenchspacing
\numberwithin{equation}{section}

\title[On the Stanley length of monomial ideals]{On the Stanley length of monomial ideals}

\author[Mircea Cimpoea\c s]{Mircea Cimpoea\c s$^1$}
\date{}

\keywords{Stanley depth; Monomial ideals; Linear quotients}

\subjclass[2020]{05E40; 06A07; 13A15; 13C15; 13P10}

\footnotetext[1]{ \emph{Mircea Cimpoea\c s}, University Politehnica of Bucharest, Faculty of
Applied Sciences, 
Bucharest, 060042, Romania and Simion Stoilow Institute of Mathematics, Research unit 5, P.O.Box 1-764,
Bucharest 014700, Romania, E-mail: mircea.cimpoeas@upb.ro,\;mircea.cimpoeas@imar.ro}

\begin{abstract}
Let $S=K[x_1,\ldots,x_n]$ be the ring of polynomials in $n$ variables over an arbitrary field $K$. Given a finitely generated multigraded module $M$,
its Stanley length, denoted by $\slength(M)$, is the minimal length of a Stanley decomposition of $M$. 

Let $I\subset S$ be a monomial ideal, minimally generated by $m$ monomials. We give an upper bound for $\slength(I)$, 
in terms of its minimal monomial generators. Also, we give precise formulas for $\slength(I)$, if $n=2$ or $m=2$.
Also, we show that if $I$ has linear quotients, then $\slength(I)=m$, and the converse holds in some special cases.
\end{abstract}

\maketitle

\section{Introduction}

Let $K$ be a field and let $S=K[x_1,x_2,\ldots,x_n]$ be the ring of polynomials in $n$ variables over $K$.
Let $M$ be a \Zng. A Stanley decomposition of $M$ is a direct sum 
$$\mathcal D: M = \bigoplus_{i=1}^rm_i K[Z_i],$$ as $K$-vector spaces, 
where $m_i\in M$ are homogeneous, $Z_i\subset\{x_1,\ldots,x_n\}$, such that $m_i K[Z_i]$ is a free $K[Z_i]$-module; $m_iK[Z_i]$ is
called a Stanley subspace of $M$. We call, the Stanley length of $\mathcal D$, the number $\slength(\mathcal D):=r$.
Also, we call the Stanley depth of $M$, the number $\sdepth(\mathcal D):=\min\limits_{i=1}^r |Z_i|$.
The Stanley depth of $M$ is
$$\sdepth(M)=\max\{\sdepth(\mathcal D)\;:\;\mathcal D\text{ is a Stanley decomposition of }M\}.$$
We define the Stanley length of $M$, the number
$$\slength(M)=\min\{\slength(\mathcal D)\;:\;\mathcal D\text{ is a Stanley decomposition of }M\}.$$
Herzog Vl\u adoiu and Zheng \cite{hvz} proved that $\sdepth(M)$ can be computed in a finite number of steps, 
when $M=I/J$, where $0\subset J\subsetneq I\subset S$ are monomial ideals.

We say that the multigraded module $M$ satisfies the Stanley inequality, if $$\sdepth(M)\geq \depth(M).$$
Stanley conjectured in \cite{stanley} that $\sdepth(M)\geq \depth(M)$,
for any \Zng $\;M$. In fact, in this form, the conjecture was stated by Apel in \cite{apel}.

Duval et al. \cite{duval} disproved this conjecture by showing a squarefree monomial ideal $I\subset S$ with
$\sdepth(S/I)<\depth(S/I)$; see \cite[Remark 3.7]{duval}. Also, they provided an example of two squarefree monomial ideals
$0\neq I \subsetneq J\subset S$ with $\sdepth(J/I)<\depth(J/I)$; see \cite[Remark 3.6]{duval}.
However, if the "Betti poset conjecture", proposed by K\"atthan in \cite{betipo}, holds, then 
$$\sdepth(I)\geq \depth(I)\text{ and }\sdepth(S/I)\geq \depth(S/I)-1,$$
for all monomial ideals $I\subset S$. 


Despite the fact that Stanley depth was extensively studied in the literature, in so far,
we don't know any paper which tackle the subject of the length of a Stanley 
decomposition of a multigraded module. We believe that the notion of the minimal length
of a Stanley decomposition, which we call it \emph{Stanley length}, could be an interesting 
measure of the combinatorial complexity of a multigraded $S$-module. The aim of our paper is 
to approach several questions regarding this new invariant.

In Proposition \ref{p1}, we show that, if
$0\to U \to M \to N \to 0$ is a short exact sequence of finitely generated multigraded $S$-modules, then 
$$\slength(M) \leq \slength(U) + \slength(N).$$
Given a monomial ideal $I\subset S$, we denote by $G(I)$, be the set of minimal monomial generators of $I$.
In Proposition \ref{p2}, we note that $\slength(I)\geq |G(I)|$.

In Proposition \ref{p4}, we prove that $\slength(I:v)\leq \slength(I)$ for any monomial ideal $I\subset S$ and any monomial $v\in S$.
Moreover, if $I=v(I:v)$, then $\slength(I:v)=\slength(I)$. 

In Proposition \ref{p5}, we prove that $\slength(I)=\slength(IS[x_{n+1}])$.
In Proposition \ref{p6}, we give some upper bounds for the Stanley length of $I\cap J$, $I+J$, $S/(I\cap J)$ and $S/(I+J)$,
 in terms of the Stanley lengths of $I$, $S/I$, $J$
and $S/J$. It is obvious that a monomial ideal $I=(u)\subset S$ is principal if and only if $\slength(I)=1$. In Proposition \ref{p7}, we 
show that if $I=(u)$ is principal, then $\slength(S/I)=\deg(u)$. We illustrate the previous results in Example \ref{exm1}.

In Theorem \ref{p28}, we give a precise formula for the Stanley length of a monomial ideal $I$ in two variables; see also Example \ref{exp28}.
Using a result from \cite{seyed}, in Theorem \ref{sey}, we give a 
criterion to compare the Stanley length of two quotients of monomial ideals in $S$. 
As a consequence, in Corollary \ref{cc1}, we prove that if
$0\subset I\subsetneq J\subset S$ are two monomial ideals and $v\in S$ is a monomial such that $(I:v)\subsetneq (J:v)$, then
$$\slength((J:v)/(I:v))\leq \slength(J/I).$$
In Corollary \ref{c2}, we prove that if $0\subset I\subsetneq J\subset S$ are two monomial ideals such 
that $\overline I\subsetneq \overline J$, where $\overline I$ and $\overline J$ are the integral closures of $I$ and $J$,
then  $$\slength(\overline J/\overline I)\leq \slength(\overline{J^k}/\overline{I^k}),\text{ for all }k\geq 1.$$
Using a method from \cite{ichim1}, in Theorem \ref{t1}, we prove that if $0\subset I\subsetneq J\subset S$ are two monomial ideals and $I^p\subset J^p$ 
are their polarizations, then $$\slength(J/I)\geq \slength(J^p/I^p).$$
Using an idea from \cite{ishaq}, in Theorem \ref{t2}, we prove that if $0\subset I\subsetneq J\subset S$ are two monomial ideals
then $$\slength(J/I)\geq\slength(\sqrt{J}/\sqrt{I}).$$
Using Janet's decomposition \cite{anwar}, in Theorem \ref{t3}, we give an upper bound for the Stanley length of a monomial $I
\subset S$, in terms of its minimal monomial generators. See also Example \ref{ext3}.

In Theorem \ref{t4} we prove that if $I$ is a monomial complete intersection ideal, generated by monomials of degrees $d_1,d_2,\ldots,d_m$, then 
$$\slength(I)\leq 1+d_1+d_1d_2+\cdots+d_1d_2\cdots d_{m-1}.$$ 
In Proposition \ref{p-trei}, we prove that if $I$ is monomial complete intersection ideal, minimally generated by three monomials of degrees $d_1, d_2$ and $d_3$, 
then we have 
$\slength(I)\leq d_1+d_2+d_3+1.$ See also Example \ref{ext4}.

Let $0\subset I\subsetneq J\subset S$ be two squarefree monomial ideals. A
Stanley decomposition $$\mathcal D:\;J/I=\bigoplus_{i=1}^r u_iK[Z_i],$$ is called squarefree, if $u_i\in S$ are squarefree monomials
and $\supp(u_i)\subset Z_i$, for all $1\leq i\leq r$. In Corollary \ref{c1},
we show that there exists a squarefree Stanley decomposition $\mathcal D$ of $J/I$ with $\slength(\mathcal D)=\slength(J/I)$.
In Proposition \ref{p33} and Proposition \ref{p4}, we give some upper bounds for $\slength(J/I)$. Also, in Proposition \ref{p35}, we give a natural lower
bound for $\slength(J/I)$. Using Theorem \ref{sey}, in Proposition \ref{p36},
 we prove that for every pair of integers $k,s\geq 1$, we have
$$ \slength(J^{(s)}/I^{(s)}) \leq \slength(J^{(ks)}/I^{(ks)}), $$
where $I^{(s)}$ and $J^{(s)}$ denote the $s$-th symbolic power of $I$ and $J$.

In Theorem \ref{t-doi}, we prove that if $I\subset S$ is a monomial ideal, minimally generated by two monomials $u_1,u_2\in S$, then:
$$\slength(I)=\min\{\deg(u_1),\deg(u_2)\}-\deg(\gcd(u_1,u_2))+1.$$
See also Example \ref{ext-doi}.

If Proposition \ref{3sq}, we prove that if $I\subset S$ is squarefree and minimally generated by three monomials, 
then $\slength(I)\leq n+1$. However, the problem to give a precise value for $\slength(I)$ seems very difficult, 
even if $I$ is also a complete intersection; see Remark \ref{trei}. 

We recall that a monomial ideal $I\subset S$ has linear quotients, if there exists $u_1 \leqslant u_2 \leqslant \cdots  \leqslant u_m$, an ordering  
on the minimal set of generators $G(I)$, such that, for any $2\leq j\leq m$, 
the ideal $(u_1,\ldots,u_{j-1}):u_j$ is generated by variables. For further details on this class of ideals, we refer the reader to \cite{jahz}.
It is easy to see that if $I$ has linear quotients, then $\slength(I)=m$. 
It is natural to ask if the converse assertion is true. In general, the answer is negative, but there are
some particular cases when the assertion holds.

Let $I_j=(u_1,u_2,\ldots,u_j)$ for all $1\leq j\leq m$.
In Theorem \ref{popo}, we prove that if $\slength(I_j)=j$, for all $1\leq j\leq m$, then $I$ has linear quotients, given by the
ordering $u_1 \leqslant u_2 \leqslant \cdots  \leqslant u_m$. In Proposition \ref{p41}, we prove that if $\slength(I)=|G(I)|\leq 3$,
then $I$ has linear quotients. Moreover, this is the best result possible, in the sense that there are monomial ideals with
$\slength(I)=|G(I)|=4$ which have no linear quotients; see Example \ref{ex1}.
Also, in Proposition \ref{p42}, we show that if $I$ is a monomial ideal in two variables with $\slength(I)=|G(I)|$, then $I$ has
linear quotients.

In Remark \ref{o44}, we note the fact that a Stanley decomposition of minimal length is not necessarily of maximal depth. 
Finally, in Theorem \ref{t5}, we give an upper bound for the Stanley length of a
squarefree monomial complete intersection, which depends on the number of minimal monomial generators. See also Example \ref{ext5}.


\section{Main results}

First, we prove the following result:

\begin{prop}\label{p1}
Let $0\to U \to M \to N \to 0$ be a short exact sequence of finitely generated multigraded $S$-modules. Then
$$\slength(M) \leq \slength(U) + \slength(N).$$
\end{prop}

\begin{proof}
As in the proof of \cite[Lemma 2.2]{asia}, given a Stanley decomposition of $U$ of length 
$r$ and a Stanley decomposition of $N$ of length $s$,
one can obtain a Stanley decomposition of $M$ of length $r+s$.
\end{proof}

\begin{prop}\label{p2}
For any monomial ideal $I\subset S$, we have that
$$\slength(I)\geq |G(I)|.$$
Moreover, $I$ is principal if and only if $\slength(I)=1$.
\end{prop}

\begin{proof}
Assume that $G(I)=\{u_1,\ldots,u_m\}$. Note that any Stanley decomposition of $I$ is of the form
\begin{equation}\label{decy}
\mathcal D: I=u_1 K[Z_1]\oplus u_2 K[Z_2] \oplus \cdots \oplus u_r K[Z_r],
\end{equation}
where $r\geq m$. This follows from the fact that a minimal monomial generator $u_i\in G(I)$ must belong to a 
Stanley space of the form $u_iK[Z_i]$. 
Moreover, there exists a unique $1\leq j\leq r$, such that $Z_j=\{x_1,\ldots,x_n\}$.

From the above, we easily get the required conclusions.
\end{proof}

\begin{prop}\label{p3}
Let $I\subset S$ be a monomial ideal and $u\in S$ a monomial. Then
$$\slength(uI)=\slength(I).$$
\end{prop}

\begin{proof}
If $\mathcal D\;:\;I=u_1K[Z_1]\oplus \cdots \oplus u_rK[Z_r]$ is a Stanley decomposition of $I$, 
then $uI=uu_1K[Z_1]\oplus \cdots \oplus uu_rK[Z_r]$ is
a Stanley decomposition of $uI$. Hence $$\slength(uI)\leq \slength(I).$$
Conversely, if $uI=v_1K[W_1]\oplus \cdots \oplus v_pK[W_p]$ is Stanley decomposition of $uI$, then 
$I=(v_1/u)K[W_1]\oplus \cdots \oplus(v_p/u)K[W_p]$ is a Stanley decomposition of $I$.
Hence $$\slength(I)\leq \slength(uI).$$
Thus, we completed the proof.
\end{proof}

\begin{prop}\label{p4}
Let $I\subset S$ be a monomial ideal and $v\in S$ be a monomial. Then 
$$\slength(I:v)\leq \slength(I).$$
The equality holds if $I=v(I:v)$.
\end{prop}

\begin{proof}
If $v\in I$ then $(I:v)=S$ and there is nothing to prove. Assume that $v\notin I$. 
By recurrence it is enough to consider the case when $v$ is a variable, let us say $v=x_n$.
From the proof of \cite[Proposition 2]{pop}, given $\mathcal D\;:\;I=u_1K[Z_1]\oplus \cdots \oplus u_rK[Z_r]$ a Stanley decomposition
of $I$, it follows that
$$\mathcal D'\;:\;(I:v) = \bigoplus_{x_n\mid u_i}(u_i/x_n) K[Z_i] \oplus \bigoplus_{x_n \nmid u_i,\;x_n\in Z_i} u_i K[Z_i]$$
is a Stanley decomposition of $(I:v)$. Obviously, $\slength(\mathcal D')\leq \slength(\mathcal D)$. 

The last assertion follows from Proposition \ref{p3}.
\end{proof}

Similarly to \cite[Lemma 3.6]{hvz}, we have the following result: 

\begin{prop}\label{p5}
Let $I\subset S$ be a monomial ideal.
Then $$\slength(IS[x_{n+1}])=\slength(I).$$
\end{prop}

\begin{proof}
If $\mathcal D\;:\;I=u_1K[Z_1]\oplus \cdots \oplus u_rK[Z_r]$ is a Stanley decomposition of $I$, then
$$\mathcal D'\;:\;IS[x_{n+1}]=u_1K[Z'_1]\oplus \cdots \oplus u_rK[Z'_r],$$ 
is a Stanley decomposition of $IS[x_{n+1}]$,
where $Z'_i=Z_i\cup\{x_{n+1}\}$ for $1\leq i\leq r$. Hence, $$\slength(IS[x_{n+1}])\leq \slength(I).$$

Conversely, let $\mathcal D'\;:\;IS[x_{n+1}]=v_1K[W_1]\oplus \cdots \oplus v_pK[W_p]$ be a Stanley decomposition of $IS[x_{n+1}]$.
Note that $$v_iK[W_i]\cap S=\begin{cases} \{0\},\; x_{n+1}\mid v_i \\ v_iK[W_i\setminus\{x_{n+1}\}],\;x_{n+1}\nmid v_i \end{cases}\text{ for all }1\leq i\leq p.$$
Hence, $\mathcal D\;:\;I= \bigoplus_{x_{n+1}\nmid v_i} v_iK[W_i\setminus\{x_{n+1}\}]$ is a Stanley decomposition for $I$ with $\slength(\mathcal D)\leq
\slength(\mathcal D')$. Hence, $$\slength(IS[x_{n+1}])\geq \slength(I),$$ as required.
\end{proof}

\begin{prop}\label{p6}
Let $I,J\subset S$ be two monomial ideals with $r=\slength(I)$, $\bar r=\slength(S/I)$, $s=\slength(J)$
and $\bar s=\slength(S/J)$. Then:
\begin{enumerate}
\item[(1)] $\slength(I\cap J)\leq rs$.
\item[(2)] $\slength(I+J)\leq \min\{ r + s\bar r,\;  s + r\bar s \}$.
\item[(3)] $\slength(S/(I+J))\leq \bar r\bar s$.
\item[(4)] $\slength(S/(I\cap J))\leq \min\{ \bar r + r\bar s,\;  \bar s + s\bar r \}$.
\end{enumerate}
\end{prop}

\begin{proof}
(1) According to the proof of \cite[Theorem 2.2(1)]{sir}, given a Stanley decomposition of $I$ of length $r$
and a Stanley decomposition of $J$ of length $s$, one can produce a Stanley decomposition of $I\cap J$ of length 
less or equal to $rs$.

(2) As in the proof of \cite[Theorem 2.2(3)]{sir}, we can write 
\begin{equation}\label{ipj}
I+J=I\oplus (J\cap(S/I)).
\end{equation}
Moreover, given a Stanley decomposition of $I$ of length $r$
and a Stanley decomposition of $J$ of length $s$, one can produce a Stanley decomposition of $J\cap(S/I)$ of length
less or equal to $rs$. Since the roles of $I$ and $J$ could be reversed, the conclusion follows from \eqref{ipj}.

(3) The proof is similar to the proof of (1).

(4) The proof is similar to the proof of (2).
\end{proof}


\begin{prop}\label{p7}
Let $u\in S$ be a monomial of degree $d \geq 1$. 
Then $$\slength(S/(u))=d.$$
\end{prop}

\begin{proof}
Assume that $u=x_1^{a_1}x_2^{a_2}\cdots x_n^{a_n}$. It is easy to check that 
\begin{align*}
\mathcal D\;:\;S/(u)=& \bigoplus_{i_1=0}^{a_1-1} x_1^{i_1}K[x_2,x_3,\ldots,x_n] \oplus \bigoplus_{i_2=0}^{a_2-1} x_1^{a_1}x_2^{i_2}
K[x_1,x_3,x_4,\ldots,x_n]\oplus \cdots \\
& \cdots \oplus \bigoplus_{i_n=0}^{a_n-1} x_1^{a_1}\cdots x_{n-1}^{a_{n-1}}x_n^{i_n}K[x_1,x_2,\ldots,x_{n-1}],
\end{align*}
is a Stanley decomposition of $S/(u)$ of length $\deg(u)$. Hence, $\slength(S/(u))\leq \deg(u)$.

In order to prove the converse inequality, let $\mathcal D\;:\;S/(u)=\bigoplus_{i=1}^r u_i K[Z_i]$ be a Stanley decomposition of $S/(u)$.
Note that $|Z_i|\leq n-1$ for all $1\leq i\leq r$. Moreover, 
for every $1\leq j\leq n$ and $0\leq \ell\leq a_j-1$, there exists a unique index $i=i(j,\ell)\in \{1,2,\ldots,r\}$, such that
$Z_i=\{x_1,\ldots,x_n\}\setminus\{x_j\}$ and $\deg_{x_j}(u_i)=\ell$. This shows that $r\geq a_1+\cdots+x_n=\deg(u)$, as required.
\end{proof}

\begin{exm}\label{exm1}\rm
(1) Let $1\leq m\leq n$ and $P=(x_1,\ldots,x_m)\subset S=K[x_1,\ldots,x_n]$, a monomial prime ideal. Then
    $$\mathcal D\;:\;I=x_1K[x_1,\ldots,x_n]\oplus x_2K[x_2,\ldots,x_n]\oplus \cdots \oplus x_mK[x_m,\ldots,x_n],$$
		is a Stanley decomposition of $I$. Since $|G(I)|=m$, from Proposition \ref{p1}, it follows that $\slength(I)=m$.
		
(2) Let $2\leq t\leq n$ and  $1\leq n_1 < n_2 < \cdots <n_t = n$ be some integers, and consider the ideal
$$I=(x_1,x_2,\ldots,x_{n_1})\cap (x_{n_1+1},x_{n_1+2},\ldots,x_{n_2})\cap \cdots \cap (x_{n_{t-1}+1},x_{n_{t-1}+2},\ldots,x_n)\subset S.$$
First, note that $I$ is minimally generated by $n_1(n_2-n_1+1)\cdots (n_t-n_{t-1}+1)$ monomials. On the other hand, using induction on $t\geq 2$, from Proposition \ref{p6}(1), it follows that $\slength(I)\leq n_1(n_2-n_1+1)\cdots (n_t-n_{t-1}+1)$. 
Hence, from Proposition \ref{p1}, it follows that
$$\slength(I)=n_1(n_2-n_1+1)\cdots (n_t-n_{t-1}+1).$$
\end{exm}


\begin{teor}\label{p28}
Let $I\subset K[x_1,x_2]$ be a monomial ideal with
$$G(I)=\{x_1^{a_1}x_2^{b_1},\;x_1^{a_2}x_2^{b_2},\ldots,x_1^{a_m}x_2^{b_m}\},$$
where $a_1>a_2>\cdots>a_m\geq 0$ and $0\leq b_1<b_2<\cdots<b_m$. Then
$$\slength(I)=\min\{a_i+b_i-a_m-b_1\;:\;1\leq i\leq m \}+1.$$
\end{teor}

\begin{proof}
Using Proposition \ref{p3} and the fact that $I=x_1^{a_m}x_2^{b_1}(I:x_1^{a_m}x_2^{b_1})$, we can assume that $a_m=b_1=0$.
Let $\mathcal D\;:\;I=\bigoplus\limits_{i=1}^r u_i K[Z_i]$ be a Stanley decomposition of $I$. Without any loss of generality, we
can assume that $Z_1=\{x_1,x_2\}$, $u_1=x_1^cx_2^d$ and $|Z_i|\leq 1$ for all $2\leq i\leq r$, where $c,d$ are some 
nonnegative integers. Since $x_1^{a_1},x_2^{b_m}\in G(I)$, it is easy to see that 
$$x_1^{\ell}K[x_2]\cap I\neq 0\text{ for all }0\leq \ell\leq c-1,\text{ and }
  x_2^{\ell}K[x_1]\cap I\neq 0\text{ for all }0\leq \ell\leq d-1.$$
Moreover, for each $0\leq \ell \leq c-1$, there exists a unique $i_{\ell}\in \{2,\ldots,r\}$ with $Z_{i_{\ell}}=\{x_2\}$ and
$$u_{i_{\ell}}K[Z_{i_{\ell}}] \subset x_1^{\ell}K[x_2]\cap I.$$
Also, for each $0\leq \ell \leq d-1$, there exists a unique $j_{\ell}\in \{2,\ldots,r\}$ with $Z_{j_{\ell}}=\{x_1\}$ and
$$u_{j_{\ell}}K[Z_{j_{\ell}}] \subset x_2^{\ell}K[x_1]\cap I.$$
This shows that $r\geq c+d+1$. Since $u_1\in I$, it follows that
$$\slength(I)\geq \min\{a_i+b_i\;:\;1\leq i\leq m \}+1.$$
In order to prove the other inequality, we fix some $1\leq i\leq m$ and we construct a Stanley decomposition of $I$ as follows:
We let $u_1=x_1^{a_i}x_2^{b_i}$. For $0\leq \ell\leq a_i-1$, we let 
$$k_{\ell}:=\min\{k\;:\;x_1^{\ell}x_2^k\in I\}\text{ and }u_{\ell+2}=x_1^{\ell}x_2^{k_{\ell}}.$$ 
It follows that $I\cap x_1^{\ell}K[x_2]=u_{\ell+2}K[x_2]$. We let $Z_{\ell+2}:=\{x_2\}$.

Similarly, for $0\leq \ell\leq a_i-1$, we let $$t_{\ell}:=\min\{t\;:\;x_1^tx_2^{\ell}\in I\}\text{ and }
u_{\ell+2+a_i}=x_1^{t_{\ell}}x_2^{\ell}.$$ 
It follows that $I\cap x_2^{\ell}K[x_1]=u_{\ell+2+a_i}K[x_1]$. We let $Z_{\ell+2}:=\{x_1\}$.
Therefore, $$I=\bigoplus_{i=1}^{a_i+b_i+1}u_iK[Z_i],$$ 
is a Stanley decomposition of $I$.
\end{proof}

\begin{exm}\label{exp28}\rm
Let $I=(x_1^5, x_1^3x_2, x_1^2x_2^3, x_2^4)\subset S=K[x_1,x_2]$. We have $m=4$, $a_1=5$, $a_2=3$, $a_3=2$, $a_4=0$, $b_1=0$,
$b_2=1$, $b_3=3$ and $b_4=4$. According to Theorem \ref{p28}, it follows that:
$$\slength(I)=\min\{a_i+b_i\;|\;1\leq i\leq 4\} - a_4 - b_1 + 1 = 4 - 0 - 0 + 1 = 5.$$
\end{exm}

We denote $\Mon(S)$, the set of monomials from $S$. The following result is a 
reinterpretation of \cite[Theorem 2.1]{seyed}:

\begin{teor}\label{sey}
Let $0\subset I_1\subsetneq J_1 \subset S$ and $0\subset I_2\subsetneq J_2 \subset S$ be monomial ideals.
Assume that there exists a function $\Phi:\Mon(S)\to\Mon(S)$, such that the following
conditions are satisfied:
\begin{enumerate}
\item[(i)] for every monomial $u\in\Mon(S)$, $u \in I_1$ if and only if $\Phi(u) \in I_2$.
\item[(ii)] for every monomial $u\in\Mon(S)$, $u \in J_1$ if and only if $\Phi(u) \in J_2$.
\item[(iii)] for every Stanley space $uK[Z] \subset S$ and every monomial $v\in\Mon(S)$, $v \in uK[Z]$ if and only if $\Phi(v) \in \Phi(u)K[Z]$.
\end{enumerate}
Then $\slength(J_1/I_1 ) \leq \slength(J_2/I_2)$.
\end{teor}

\begin{proof}
Consider a Stanley decomposition 
$$\mathcal D\;:\;J_2/I_2=\bigoplus_{i=1}^r v_iK[Z_i]$$
of $J_2/I_2$ such that $\slength(J_2/I_2)=\slength(\mathcal D)$. For each $1\leq i\leq r$, let
$$U_i=\{u\in J_1\setminus I_1\;:\;u\text{ is a monomial with }\Phi(u)\in v_iK[Z_i]\}.$$
Without any loss of generality, we may assume that there exists $1\leq \ell\leq r$ such that, 
$U_i=\emptyset$ if and only if $\ell+1\leq i\leq r$. For $1\leq i\leq \ell$, let $u_i$ be the greater common divisor
of the monomials from $U_i$. According to the proof of \cite[Theorem 2.1]{seyed},
$$\mathcal D'\;:\;J_1/I_1=\bigoplus_{i=1}^{\ell} u_iK[Z_i]$$
is a Stanley decomposition of $J_1/I_1$. Hence
$$\slength(J_1/I_1)\leq \slength(\mathcal D')=\ell\leq r=\slength(J_2/I_2),$$
which completes the proof.
\end{proof}

Now, we are able to prove the following corollary, which extends the first assertion from Proposition \ref{p4}:

\begin{cor}\label{cc1}
Let $0\subset I\subsetneq J\subset S$ be two monomial ideals and $v\in S$ a monomial, such that $(I:v)\subsetneq (J:v)$.
Then $$\slength((J:v)/(I:v))\leq \slength(J/I).$$
\end{cor}

\begin{proof}
Similar to the proof of \cite[Proposition 2.5]{seyed}, we let 
$$\Phi:\Mon(S)\to\Mon(S),\; \Phi(u)=uv,\text{ for all }u\in\Mon(S).$$
It is clear that, given a monomial $u\in S$, $u\in (J:v)\setminus (I:v)$ if and only if $\Phi(u)=uv\in J\setminus I$.
Also, if $w\in\Mon(S)$ then $w\in uK[Z]$ if and only if $\Phi(w)=wv\in uvK[Z]=\Phi(u)K[Z]$, for any $Z\subset \{x_1,\ldots,x_n\}$.
Hence, the conclusion follows from Theorem \ref{sey}.
\end{proof}

Given a monomial ideal $I\subset S$, its integral closure is the monomial ideal $\overline{I}$, generated by the monomials $u\in S$,
such that there exists $s\geq 1$ with $u^s\in I^s$; see for instance \cite[Theorem 1.4.2]{hh}. 
Using the arguments from \cite[Proposition 2.3]{seyed}
and \cite[Proposition 2.4]{seyed}, we deduce the following:

\begin{cor}\label{c2}
Let $0\subset I\subsetneq J\subset S$ be two monomial ideals such that $\overline I\subsetneq \overline J$. Then:
\begin{enumerate}
\item For every integer $k\geq 1$, we have $\slength(\overline J/\overline I)\leq \slength(\overline{J^k}/\overline{I^k})$.
\item There exists an integer $k\geq 1$, such that $\slength(\overline J/\overline I)\leq \slength(J^{ks}/I^{ks})$, for all $s\geq 1$.
\end{enumerate} 
\end{cor}

We recall the definition of the polarization of a monomial ideal, given in \cite{hh}.
Let $I\subset S$ be a monomial ideal with $G(I)=\{u_1,\ldots,u_m\}$, where $u_i=\prod\limits_{j=1}^n x_j^{a_{ij}}$ for $1\leq i\leq m$.
For each $j$, let $a_j=\max\{a_{ij}\;:\;1\leq i\leq m\}$. Set $a=(a_1,\ldots,a_n)$. We choose $g=(g_1,\ldots,g_n)\in \mathbb N^n$
such that $a_j\leq g_j$ for all $1\leq j\leq n$, and set $S^p$ to be the polynomial ring
$$S^p = K[x_{jk}\;:\;1\leq j\leq n,\;1\leq k\leq g_j].$$
Then the polarization of $I$ is the squarefree monomial ideal $I^p\subset S^p$ generated by $v_1,\ldots,v_m$, where
$$v_i=\prod_{j=1}^n\prod_{k=1}^{a_{ij}}x_{jk},\;1\leq i\leq m.$$
The polarization can be done also step by step. The first step polarization of $I$, with respect to $x_j$, is the
ideal $I^1\subset S[y]$, generated by $v_1,\ldots,v_m$, where
$$v_i=\begin{cases} yu_i/x_j,&x_j^2\mid u_i \\ u_i,& x_j^2\nmid u_i \end{cases}.$$
This procedure should be used $g_1+g_2+\cdots+g_n-n$ times, in order to obtain $I^p$.


\begin{teor}\label{t1}
Let $0\subset I\subsetneq J\subset S$ be two monomial ideals and let $0\subset I^p\subsetneq J^p\subset S^p$ be their polarizations.
Then $$\slength(J/I) \geq \slength(J^p/I^p).$$
\end{teor}

\begin{proof}
Since $I^p$ and $J^p$ can be obtained by repeating the $1$-step polarization several times, it is enough to prove that
$\slength(J/I)\geq \slength(J^1/I^1)$,
where $I^1,J^1\subset S^1=S[y]$ are the first step polarizations of $I$ and $J$, with respect to, let us say $x_n$.

We consider the map $\Phi:\Mon(S)\to\Mon(S[y])$, defined by
$$\Phi(u)=\begin{cases} yu/x_n,&x_n^2\mid u \\ u,& x_n^2\nmid u \end{cases},\text{ for all monomials }u\in S.$$
Let $\mathcal D\;:\;J/I=\bigoplus\limits_{i=1}^r u_i K[Z_i]$ be a Stanley decomposition of $J/I$ with 
$$\slength(\mathcal D)=\slength(J/I).$$
We consider the subsets of variables
$$Z'_i=\begin{cases} Z_i\cup\{y\},&x_n\in Z_i \\ Z_i\cup\{ \{x_n,y\}\setminus\{ \Phi(x_nu_i)/\Phi(u_i) \} \},& x_n\notin Z_i  \end{cases},
\text{ for all }1\leq i\leq r.$$
Note that $Z'_i=Z_i\cup\{y\}$, with the exception of the case $x_n\notin Z_i$, $x_n\mid u_i$ and $x_n^2\nmid u_i$, when $Z'_i=Z_i\cup\{x_n\}$.

From the proof of \cite[Theorem 4.3]{ichim1}, it follows that
$$\mathcal D'\;:\;J^1/I^1 = \bigoplus_{i=1}^r \Phi(u_i)K[Z'_i]$$
is a Stanley decomposition of $J^1/I^1$. Hence $\slength(J/I) \geq \slength(J^1/I^1)$, as required.
\end{proof}

\begin{teor}\label{t2}
Let $0\subset I\subsetneq J\subset S$ be two monomial ideals. Then
$$\slength(J/I)\geq \slength(\sqrt{J}/\sqrt{I}).$$
\end{teor}

\begin{proof}
Let $\mathcal D\;:\;J/I=\bigoplus\limits_{i=1}^r u_iK[Z_i]$ be a Stanley decomposition of $J/I$ with $$\slength(\mathcal D)=\slength(J/I).$$
Let $a\in\mathbb N$ be the maximal degree of a variable $x_j$ with $x_j^a\mid u$  for some $u\in G(I)\cup G(J)$. 
This means that if $v\in G(\sqrt{I})\cup G(\sqrt{J})$ then $v^a\in I+J$. We consider the set of indexes
$$\mathcal I = \{i\in\{1,2,\ldots,r\}\;:\;\text{ If }x_k\notin Z_i,\text{ then }a\mid \deg_{x_k}(u_i)\},$$
where $\deg_{x_k}(u_i)=\max\{b\;:\;x_k^b\mid u_i\}$. Let $c_{ij}=\left\lceil \frac{b_{ij}}{a} \right\rceil$, for all $i\in\mathcal I$ and $1\leq j\leq n$. 
Let $$u'_i=x_{1}^{c_{i1}}x_2^{c_{i2}}\cdots x_n^{c_{in}},\text{ for all }i\in\mathcal I.$$
According to the proof of \cite[Theorem 2.1]{ishaq}, it follows that
$$\mathcal D'\;:\;\sqrt{J}/\sqrt{I} = \bigoplus_{i\in\mathcal I} u'_iK[Z_i],$$
is a Stanley decomposition of $\sqrt{J}/\sqrt{I}$. Since $\mathcal I\subset\{1,2,\ldots,r\}$, we get the required conclusion.
\end{proof}

\begin{obs}\rm
Another way to proof Theorem \ref{t2} is to apply Theorem \ref{sey} and to use the arguments from the proof of
\cite[Proposition 2.2]{seyed}.
\end{obs}

For all $1\leq k\leq n$, 
we consider the numbers
\begin{align*}
& \alpha_k(I)=\max\{j\;:\;x_k^j\mid u\text{ for all }u\in G(I)\}\text{ and }\\
& \beta_k(I)=\max\{j\;:\;x_k^j\mid u\text{ for some }u\in G(I)\}.
\end{align*}

\begin{teor}\label{t3}
Let $0\neq I\subset S$ be a monomial ideal. With the above notations, we have that
$$\slength(I)\leq \prod_{k=2}^n (\beta_k(I)-\alpha_k(I)+1).$$
\end{teor}

\begin{proof}
We proceed by induction on $n\geq 1$. If $n=1$, then $I=(x_a^b)$, for some nonnegative integer $b$. It follows that
$\alpha_1=\beta_1=b$ and thus $1=\slength(I)=\beta_1-\alpha_1+1=1$, as required. 

Assume $n\geq 2$. We consider the Janet's decomposition of $I$; see \cite{anwar} for further details. 
Let $S'=K[x_1,\ldots,x_{n-1}]$. For $j\geq 0$, we consider the
monomial ideals $I_j\subset S'$ defined by $I\cap x_n^jS' = x_n^jI_j$. Let $\alpha=\alpha_n(I)$ and $\beta=\beta_n(I)$. We decompose $G(I)$ as a disjoint
union $$G(I)=G_{\alpha}(I)\cup G_{\alpha+1}(I) \cup \cdots \cup G_{\beta}(I),$$
where $G_j(I)=\{u\in G(I)\;:\;\deg_{x_n}(u)=j\}$. With the above notations, we have that
$$I_j=(G_{\alpha}\cup \cdots \cup G_j):x_n^j.$$
Note that $I_j=0$ for $j<\alpha$ and $I_{\beta}=I_{\beta+1}=\cdots$. According to \cite[Lemma 2.1]{anwar}, we can write
$$I=\bigoplus_{\alpha\leq j<\beta}x_n^j I_j \oplus x_n^{\beta}I_{\beta}[x_n].$$
Moreover, if $\mathcal D_j\;:\;I_j=\bigoplus\limits_{i=1}^{r_j} u_{ij}K[Z_{ij}]$, for $\alpha\leq j\leq \beta$, are Stanley decompositions, then
\begin{equation}\label{artdeco}
\mathcal D\;:\;I=\bigoplus_{\alpha\leq j<\beta} \bigoplus_{i=1}^{r_j} x_n^ju_{ij}K[Z_{ij}]\oplus 
\bigoplus_{i=1}^{r_{\beta}}u_{ij}K[Z_{ij}\cup\{x_n\}],
\end{equation}
is a Stanley decomposition of $I$. From the definition of the ideals $I_j$'s, it is easy to note that
\begin{equation}\label{artdeco2}
\alpha_k(I) \leq \alpha_k(I_j) \leq \beta_k(I_j)\leq \beta_k(I)\text{ for all }\alpha\leq j\leq \beta.
\end{equation}
Now, the conclusion follows from \eqref{artdeco}, \eqref{artdeco2} and the induction hypothesis.
\end{proof}

\begin{exm}\label{ext3}
Let $I=(x_1^3x_2x_3^2,\;x_1x_2^4,\;x_2^2x_3^5)\subset K[x_1,x_2,x_3]$. It is easy to see that
\begin{align*}
& \alpha_2(I)=\max\{j\;:\;x_2^j\mid u\text{ for all }u\in G(I)\}=1\text{ and }\\
& \beta_2(I)=\max\{j\;:\;x_2^j\mid u\text{ for some }u\in G(I)\}=4.
\end{align*}
Similarly, we have $\alpha_3(I)=0$ and $\alpha_3(I)=5$. From Theorem \ref{t3} it follows that
$$\slength(I)\leq (\beta_2(I)-\alpha_2(I)+1)\cdot (\beta_3(I)-\alpha_3(I)+1) = 4 \cdot 6 = 26.$$
\end{exm}

\begin{lema}\label{lem}
Let $I\subset S$ be a monomial ideal with $G(I)=\{u_1,u_2,\ldots,u_m\}$. Assume that $x_n\mid u_1$ and $x_n\nmid u_i$ for all $2\leq i\leq n$.
We consider the ideals $I'=(u_1/x_1,u_2,\ldots,u_m)\subset S$ and $I''=(u_2,\ldots,u_m)\subset K[x_1,\ldots,x_{n-1}]$. Then:
$$\slength(I)\leq \slength(I')+\slength(I'').$$
\end{lema}

\begin{proof}
Since $(I:x_n)=I'$ and $(I,x_n)=(I'',x_n)$, it follows that we have the decomposition
$$I = I'' \oplus x_nI'.$$
Now, the conclusion follows from Proposition \ref{p1} and Proposition \ref{p3}.

\end{proof}

\begin{lema}\label{lem2}
Let $I'\subset S':=K[x_1,\ldots,x_{n-1}]$ be a monomial ideal. We consider the ideal $I=(I',x_n)\subset S$.
Then
$$\slength(I) = \slength(I')+1.$$
\end{lema}

\begin{proof}
Since $I=x_nS \oplus I'$, from Proposition \ref{p1}, it follows that $$\slength(I)\leq \slength(I')+1.$$
In order to prove the other inequality, let $\mathcal D\;:\;I=\bigoplus\limits_{i=1}^r u_iK[Z_i]$ be a Stanley
decomposition of $I$ with $\slength(\mathcal D)=\slength(I)$. 
Note that $$u_iK[Z_i]\cap S'=\begin{cases} 0,& x_n\mid u_i \\ u_iK[Z_i\setminus\{x_n\}],& x_n\nmid u_i \end{cases}.$$
It follows that
$$I' = I\cap S' = \bigoplus_{i=1}^r (u_iK[Z_i] \cap S').$$
Since $x_n\in G(I)$, without any loss of generality, we can assume that $u_1=x_n$. Hence
$$I' = \bigoplus_{i=2}^r (u_iK[Z_i] \cap S'),$$
is a Stanley decomposition of $I'$ of length $\leq r-1$. Hence $$\slength(I) \geq \slength(I')+1,$$ as required.
\end{proof}

\begin{teor}\label{t4}
Let $I=(u_1,u_2,\ldots,u_m)\subset S$ be a monomial complete intersection and let $d_i:=\deg(u_i)$, for $1\leq i\leq m$.
Assume that $d_1\leq d_2\leq \cdots \leq d_n$. Then
$$\slength(I)\leq 1+d_1+d_1d_2+\cdots+d_1d_2\cdots d_{m-1}.$$
\end{teor}

\begin{proof}
We denote $\Phi_m(d_1,d_2,\ldots,d_m)=1+d_1+d_1d_2+\cdots+d_1d_2\cdots d_{m-1}$. 
Our aim is to prove that $\slength(I)\leq \Phi_m(d_1,d_2,\ldots,d_m)$.
We proceed by induction on $m$. If $m=1$, that is $I=(u_1)$, then 
$\slength(I)=1\leq \Phi_1(d_1)=1$, as required. 

Now, assume $m\geq 2$. Without any loss of generality, we can assume that $x_n\mid u_1$. Since $I$ is a
monomial complete intersection, it follows that $x_n\nmid u_j$, for all $2\leq j\leq m$. We consider two cases:
\begin{enumerate}
\item[(i)] $d_1=1$. We have that $u_1=x_n$ and $u_2,\ldots,u_m\in S'=K[x_1,\ldots,x_{n-1}]$. Let $I'=(u_2,\ldots,u_m)\subset S'$. From Lemma \ref{lem2}
it follows that
$$\slength(I) \leq 1+\Phi_{m-1}(d_2,\ldots,d_m)\leq 1+ d_2 + d_2d_3 + \cdots + d_2d_3\cdots d_{m-1},$$
as required. 
\item[(ii)] $d_1\geq 2$. 
Let $u'_1:=u_1/x_n$. We consider
the ideals $I'=(u'_1,u_2,\ldots,u_m)\subset S$ and $I''=(u_2,\ldots,u_m)\subset K[x_1,\ldots,x_{n-1}]$. 
Note that $I'$ and $I''$ are monomial complete intersections.
Hence, from Lemma \ref{lem} and the induction hypothesis, it follows that
\begin{align*}
\slength(I) &\leq  \slength(I')+\slength(I'')  \\
 & \leq \Phi_{m}(d_1-1,d_2,\ldots,d_m)+\Phi_{m-1}(d_2,\ldots,d_m) \\
& = 1+(d_1-1)(1+d_2+ \cdots + d_2\cdots d_{m-1}) \\
& \hspace{10pt} + (1+d_2+ \cdots + d_2\cdots d_{m-1}) \\
& = 1 + d_1 +d_1d_2 + \cdots + d_1d_2 \cdots d_{m-1}, 
\end{align*}
as required. 
\end{enumerate} Hence, the proof is complete.
\end{proof}

In general, this bound is not optimal, as the following proposition shows:

\begin{prop}\label{p-trei}
Let $I \subset S$ be a monomial complete intersection, minimally generated by $3$ monomials of degrees $d_1$, $d_2$ and $d_3$. 
Then $$\slength(I)\leq d_1+d_2+d_3+1.$$
\end{prop}

\begin{proof}
Since $I$ is a monomial complete intersection, it follows that $I$ can be obtained by applying several partial polarizations to
the ideal $J=(x_1^{d_1},x_2^{d_2},x_3^{d_3})\subset S=K[x_1,x_2,x_3]$. Therefore, from Theorem \ref{t1}, we can assume that $I=J$.
On the other hand, $I$ has the following Stanley decomposition
\begin{align*}
\mathcal D\;:\;I=& x_1^{d_1}x_2^{d_2}x_3^{d_3}K[x_1,x_2,x_3]
                   \oplus \bigoplus_{i_1=0}^{d_1-1}x_1^{i_1}x_2^{d_2}K[x_2,x_3] \\
									 & \oplus \bigoplus_{i_2=0}^{d_2-1}x_2^{i_2}x_3^{d_3}K[x_1,x_3]
									 \oplus \bigoplus_{i_3=0}^{d_3-1}x_3^{i_3}x_1^{d_1} K[x_1,x_2],
\end{align*}									
with $\slength(\mathcal D)=d_1+d_2+d_3+1$. Thus, the proof is complete.
\end{proof}

\begin{exm}\label{ext4}
Let $I=(x_1^3,x_2^2x_3^2,x_4x_5^3)\subset S=K[x_1,\ldots,x_5]$. According to Theorem \ref{t4}, we have that
$$\slength(I)\leq 1+3+3\cdot 4 = 16.$$
On the other hand, according to Proposition \ref{p-trei}, we have 
$$\slength(I)\leq 3+4+4+1=12.$$
\end{exm}


\section{The squarefree case}

Let $[n]=\{1,2,\ldots,n\}$. For a subset $F\subset [n]$, we denote $x_F=\prod\limits_{j\in F}x_j$.
Let $I\subset S$ be a squarefree monomial ideal. The Stanley-Reisner simplicial complex associated to $S/I$ is
$$\Delta(S/I):=\{F\subset [n]\;:\;x_F\notin I\}.$$
Let $0\subset I\subsetneq J\subset S$ be two squarefree monomial ideals. The relative Stanley-Reisner simplicial complex associated to $J/I$ is
$$\Delta(J/I):=\Delta(S/I)\setminus\Delta(S/J)=\{F\subset [n]\;:\;x_F\in J\setminus I\}.$$
An element $F\in\Delta(J/I)$, maximal with respect to inclusion, is called a facet.

Given $F\subset G\subset [n]$, the set $[F,G]=\{H\;:\;F\subset H\subset G\}$ is called an interval.

An interval partition of $\Delta(J/I)$ is a decomposition
$$\mathcal P\;:\;\Delta(J/I)=\bigcup_{i=1}^r [F_i,G_i],$$
of $\Delta(J/I)$, as a disjoint union of intervals.

For a monomial $u\in S$, the support of $u$, denoted by $\supp(u)$, is the set 
$$\supp(u)=\{x_j\;:\;x_j\mid u\}.$$
A Stanley decomposition $$\mathcal D\;:\; J/I=\bigoplus\limits_{i=1}^r u_iK[Z_i],$$ of $J/I$ is called squarefree,
 if $u_i$'s are squarefree
and $\supp(u_i)\subset Z_i$, for all $1\leq i\leq r$.

As a particular case of \cite[Theorem 2.1]{hvz}, we have the following:

\begin{teor}\label{sqt}
Let $0\subset I\subsetneq J\subset S$ be two squarefree monomial ideals.
\begin{enumerate}
\item[(1)] Let $\mathcal P\;:\;\Delta(J/I)=\bigcup\limits_{i=1}^r [F_i,G_i]$ be an interval partition of $\Delta(J/I)$. 
           Let $u_i=x_{F_i}$ and $Z_i=\supp(x_{G_i})$, for $1\leq i\leq r$. Then $$\mathcal D\;:\;J/I=\bigoplus_{i=1}^r u_i K[Z_i],$$
					 is a squarefree Stanley decomposition of $J/I$.
\item[(2)] Let $\mathcal D\;:\;J/I=\bigoplus\limits_{i=1}^r u_i K[Z_i]$ be an arbitrary Stanley decomposition of $J/I$. Then
           $$\mathcal D'\;:\;J/I= \bigoplus_{u_i\text{ is squarefree}}u_i K[Z_i\cup\supp(u_i)],$$
					 is a squarefree Stanley decomposition of $J/I$. 
\end{enumerate}
\end{teor}

\begin{cor}\label{c1}
Let $0\subset I\subsetneq J\subset S$ be two squarefree monomial ideals. Then:
\begin{enumerate}
\item[(1)] There exists a squarefree Stanley decomposition $\mathcal D$ of $J/I$ with $\sdepth(J/I)=\sdepth(\mathcal D)$.
\item[(2)] There exists a squarefree Stanley decomposition $\mathcal D$ of $J/I$ with $\slength(J/I)=\slength(\mathcal D)$.
\end{enumerate}
\end{cor}

\begin{proof}
In the statement (2) of Theorem \ref{sqt}, we have that
$$\sdepth(\mathcal D')\geq \sdepth(\mathcal D)\text{ and }\slength(\mathcal D')\leq \slength(\mathcal D).$$
Hence, we get the required conclusions.
\end{proof}

We denote $2^{[n]}$, the family of all subsets of $[n]$. Let $P\subset 2^{[n]}$. An interval partition of $P$ is a
disjoint union
$$\mathcal P\;:\;P=\bigcup_{i=1}^r [F_i,G_i],$$
where $F_i\subset G_i\subset [n]$. We define the Stanley length of $\mathcal P$ to be $\slength(\mathcal P)=r$.
Also, the Stanley length of $P$ is
$$\slength(P):=\min\{\slength(\mathcal P)\;:\;\mathcal P\text{ is an interval partition of }P\}.$$
It is clear that $\slength(P)=1$ if and only if $P$ is a set interval.

As a direct consequence of Corollary \ref{c1}, we have:

\begin{prop}\label{propi}
Let $0\subset I\subsetneq J\subset S$ be two squarefree monomial ideals. Then
$$\slength(J/I)=\slength(\Delta(J/I)).$$
\end{prop}

The following proposition, gives a trivial upper bound for the Stanley length of a quotient of two squarefree monomial ideals:

\begin{prop}\label{p33}
Let $0\subset I\subsetneq J\subset S$ be two squarefree monomial ideals. Then
$$\slength(J/I)\leq |\Delta(J/I)|.$$
\end{prop}

\begin{proof}
It follows from Corollary \ref{c1}, Proposition \ref{propi} and the trivial partition $$\Delta(J/I)=\bigcup\limits_{F\in \Delta(J/I)}[F,F].$$
\end{proof}

However, one can give a better upper bound, if the Stanley depth is known:

\begin{prop}\label{p34}
Let $0\subset I\subsetneq J\subset S$ be two squarefree monomial ideals, such that $\sdepth(J/I)=s$. Then
$$\slength(J/I)\leq |\{F\in \Delta(J/I)\;:\;|F|\geq s\}|.$$
\end{prop}

\begin{proof}
Since $\sdepth(J/I)=s$, from Corollary \ref{c1}, there exists a partition 
$$\mathcal P\;:\;\Delta(J/I)=\bigcup\limits_{i=1}^r[F_i,G_i]$$ of $\Delta(J/I)$ with $|G_i|\geq s$, 
for all $1\leq i\leq r$. Hence, the conclusion follows form Proposition \ref{propi}.
\end{proof}

\begin{prop}\label{p35}
Let $0\subset I\subsetneq J\subset S$ be two squarefree monomial ideals and let $\Delta=\Delta(J/I)$.
Then $\slength(J/I)$ is greater or equal to the number of facets of $\Delta$.
\end{prop}

\begin{proof}
Give an interval partition of $\Delta$, i.e. $\Delta=\bigcup\limits_{i=1}^r [F_i,G_i]$, the set of facets of 
$\Delta$ is included in $\{G_1,\ldots,G_r\}$. Hence $r$ is greater or equal to the number of facets of $\Delta$.
The conclusion follows from Proposition \ref{propi}.
\end{proof}

Let $I\subset S$ be a squarefree monomial ideal and suppose that $I$ has the irredundant
primary decomposition $$I=P_1\cap P_2\cap \cdots \cap P_m,$$
where $P_i$'s are prime monomial ideals, that is $P_i$'s are generated by subsets of variables.

The $k$-th symbolic power of $I$, denoted by $I^{(k)}$, is defined to be
$$I^{(k)}=P_1^k\cap P_2^k\cap \cdots \cap P_m^k.$$
We have the following result:

\begin{prop}\label{p36}
Let $0\subset I\subsetneq J\subset S$ be two squarefree monomial ideals. Then, for every pair of integers $k,s\geq 1$, we have
$$ \slength(J^{(s)}/I^{(s)}) \leq \slength(J^{(ks)}/I^{(ks)}). $$
\end{prop}

\begin{proof}
As in the proof of \cite[Theorem 3.1]{seyed}, we let $\Phi:\Mon(S)\to\Mon(s)$, $\Phi(u)=u^k$, and we apply Theorem \ref{sey}.
\end{proof}

\begin{prop}
Let $1\leq t\leq n$ be an integer. Let $I\subset S$ be a squarefree monomial ideal and assume that there exists a subset 
$A\subset \{x_1,\ldots,x_n\}$, such that, for any $P\in \Ass(S/I)$, we have $|P\cap A|=t$. Then, for any $k\geq 1$, we have
$$\slength(I^{(k)})\leq \slength(I^{(k+t)})\text{ and }\slength(S/I^{(k)})\leq \slength(S/I^{(k+t)}).$$
In particular, the conclusion holds, if $I$ is unmixted of height $d$.
\end{prop}

\begin{proof}
We use the argument from \cite[Remark 3.3]{seyed} and we apply Theorem \ref{sey}.

If $I$ is unmixted of heigth $d$, then, $I=P_1\cap\cdots \cap P_m$ such that the heigth of $P_i$ is $d$, for all $1\leq i\leq d$. 
Hence, the initial hypothesis holds for $A=\{x_1,\ldots,x_n\}$. See also \cite[Theorem 3.7]{seyed}.
\end{proof}

\begin{teor}\label{t-doi}
Let $I\subset S$ be a monomial ideal, minimally generated by two monomials $u_1,u_2\in S$. Then:
$$\slength(I)=\min\{\deg(u_1),\deg(u_2)\}-\deg(\gcd(u_1,u_2))+1.$$
\end{teor}

\begin{proof}
Let $v=\gcd(u_1,u_2)$. Since $I=v(I:v)$, by Proposition \ref{p4}, we can assume that $v=1$ and, thus, $I$ is
a monomial complete intersection. For simplicity, we also assume that $1\leq d_1=\deg(u_1)\leq \deg(u_2)$.

From Theorem \ref{t4} it follows that $\slength(I)\leq d_1+1$. 
We consider $I^p$, the polarization of $I$. Since, by Theorem \ref{t1}, we have $\slength(I)\geq \slength(I^p)$,
in order to prove that $\slength(I)\leq d_1+1$, we can assume that $I$ is squarefree.
Let $$\mathcal P\;:\;\Delta(I)=\bigcup_{i=1}^r [F_i,G_i],$$ be a partition of $\Delta(I)$.
Without any loss of generality, we can assume that $G_1=[n]$. Let $I_0=(x_1x_2\cdots x_n)$.
If $F_1=[n]$, then 
$$\mathcal P'\;:\;\Delta(I/I_0)=\bigcup_{i=2}^r [F_i,G_i],$$ is a partition of $\Delta(I/I_0)$.
Since $\Delta(I)$ contains all the subsets of $[n]$ of cardinality $n-1$, it follows that 
$\Delta(I/I_0)$ has $n$ facets. Hence, by Proposition \ref{p34}, we have that 
$r-1\geq \slength(I/I_0) \geq n$.
Thus $r\geq n = d_1+d_2 > d_1+1$. 

Now, assume that $F_1\subsetneq [n]$. Since $|F_1|\geq d_1$, the set
$$\{F\subset [n]\;:\;|F|=n-1,\;F\in [F_1,G_1]\},$$ has at most $n-d_1$ elements. Hence, 
$\Delta(S/I)$ contains at least $d_1$ subsets with $n-1$ elements which are not contained in $[F_1,G_1]$. This implies $r\geq d_1+1$,
as required. Thus, the proof is complete.
\end{proof}

\begin{exm}\label{ext-doi}\rm
Let $I=(x_1^2x_2^3x_3,x_2^2x_3^2x_4)\subset S=K[x_1,x_2,x_3,x_4]$. We have $u_1=x_1^2x_2^3x_3$, $u_2=x_2^2x_3^2x_4$ and thus
$\gcd(u_1,u_2)=x_2^2x_3$. According to Theorem \ref{t-doi} we have
$\slength(I)
=\min\{6,5\}-3+1=3$.
\end{exm}

\begin{prop}\label{3sq}
Let $I$ be a squarefree monomial ideal, minimally generated by three monomials. 
Then $\slength(I)\leq n+1$.
\end{prop}

\begin{proof}
According to \cite[Theorem 4.1]{shen}, we have that $\sdepth(I)=n-1$. Hence, there exists a partition
$$\mathcal P\;:\;\Delta(I)=\bigcup_{i=1}^r [F_i,G_i],$$
where $|G_i|\geq n-1$ for all $1\leq i\leq r$. Since $[n]$ has $n+1$ subsets with at least $n-1$ elements, we get the required conclusion.
\end{proof}

\begin{obs}\label{trei}\rm
It would be interesting to find a formula for $\slength(I)$, where $I\subset S$ is a monomial ideal with $G(I)=\{u_1,u_2,u_3\}$.
However, this seems very difficult, even if we assume that $I$ is a complete intersection. Let $d_i=\deg(u_i)$, $1\leq i\leq 3$,
and assume that $d_1\leq d_2\leq d_3$. According to Theorem \ref{t4} and Proposition \ref{p-trei}, it follows that
$$\slength(I)\leq \min\{d_1+d_1d_2,d_1+d_2+d_3\}+1.$$
We can conjecture that $\slength(I) = \min\{d_1+d_1d_2,d_1+d_2+d_3\}+1$. In order to tackle this problem, using Theorem \ref{t1}, it is enough
to assume that $I$ is squarefree and, moreover, that $d_1+d_2+d_3=n$ and 
$$u_1=x_1\cdots x_{d_1},\; u_2=x_{d_1+1}\cdots x_{d_1+d_2}\text{ and }u_3=x_{d_1+d_2+1}\cdots x_{n}.$$
Given an arbitrary interval partition $$\mathcal P\;:\;\Delta(I)=\bigcup\limits_{i=1}^r [F_i,G_i],$$
of $\Delta(I)$, we would have to show that $r\geq \min\{d_1+d_1d_2,n\}+1$. If $d_1=1$ them, according to Lemma \ref{lem2}, we have
$$r\geq \slength(I)=1+\slength((u_2,u_3)K[x_2,\ldots,x_n])=1+d_2,$$
and we are done. Hence, we can assume that $d_1\geq 2$. Also, if $n>d_1+d_1d_2$ then, by Proposition \ref{p4} and Proposition \ref{p5}, we have that
\begin{align*}
\slength(I) &\geq \slength((I:x_{d_1+d_1d_2+1}\cdots x_n))\\
&= \slength((I:x_{d_1+d_1d_2+1}\cdots x_n)\cap K[x_1,\ldots,x_{d_1+d_1d_2}]).
\end{align*}
Hence, we can reduce to the case that $n\leq d_1+d_1d_2$ and we have to show that $r\geq n+1$.

Unfortunately, we were not able, neither to prove this, nor to find a counterexample.
\end{obs}

\section{Monomial ideals with linear quotients}

We recall that a monomial ideal $I\subset S$ has linear quotients, if there exists a linear order
 $$u_1 \leqslant u_2 \leqslant \cdots  \leqslant u_m,$$
 on $G(I)$, that is,
for every $2 \leq j \leq m$, the ideal $(u_1, \ldots , u_{j-1}) : u_j$ is generated by a subset of $n_j$ variables.
We let $I_j:=(u_1,\ldots,u_j)$, for $1\leq j\leq m$. 
Let $Z_1=\{x_1,\ldots,x_n\}$ and 
$Z_j=\{x_i\;:\;x_i\notin (I_{j-1} : u_j)\}$ for $2\leq j\leq m$.

Note that, for any $2\leq j\leq m$, we have
$$I_j/I_{j-1} = u_j (S/(I_{j-1}:u_j)) = u_j K[Z_j].$$
Hence, the ideal $I$ has the Stanley decomposition
\begin{equation}\label{sdep}
I = u_1K[Z_1]\oplus u_2K[Z_2] \oplus \cdots \oplus u_mK[Z_m].
\end{equation}
In particular, from Proposition \ref{p1}, it follows that $\slength(I)=|G(I)|$.

\begin{teor}\label{popo}
Let $I\subset S$ be a monomial ideal, minimally generated be $m$ monomials with $G(I)=\{u_1,\ldots,u_m\}$. 
Then the following are equivalent:
\begin{enumerate}
\item There exists a linear order $u_1 \leqslant u_2 \leqslant \cdots  \leqslant u_m$.
\item The ideal $I_j=(u_1,u_2,\ldots,u_j)$ has a Stanley decompositions of the form
      $$I_j=u_1K[Z_1]\oplus u_2K[Z_2]\oplus \cdots \oplus u_jK[Z_j],\text{ for all }1\leq j\leq m.$$
\end{enumerate}
\end{teor}

\begin{proof}
$(1)\Rightarrow (2)$. It follows from the above discussion. 

$(2)\Rightarrow (1)$. We proceed by induction on $m$. The case $m=1$ is trivial. Assume $m\geq 2$. By induction hypothesis, it follows that
$u_1 \leqslant u_2 \leqslant \cdots  \leqslant u_{m-1}$ is a linear order on $I_{m-1}$. In order to prove that $I$ has linear quotients, 
it is enough to show that 
$$(I_{m-1}:u_m) = (x_j\;:\;x_j\notin Z_m).$$
Indeed, if $x_j\notin Z_m$, then $x_ju_m\notin u_mK[Z_m]$ and therefore, from the hypothesis 
$$I=I_m=u_1K[Z_1]\oplus u_2K[Z_2]\oplus \cdots \oplus u_mK[Z_m],$$
it follows that
$$x_ju_m \in u_1K[Z_1]\oplus \cdots \oplus u_{m-1}K[Z_{m-1}] = I_{m-1}.$$
Hence $(x_j\;:\;x_j\notin Z_m)\subset (I':u_m)$. In order to prove the converse inclusion, assume that
there exists $0\neq w\in K[Z_m]$ such that $w\in (I':u_m)$. It follows that 
$$wu_m \in u_mK[Z_m] \cap I' = u_mK[Z_m]\cap (u_1K[Z_1]\oplus \cdots \oplus u_{m-1}K[Z_{m-1}]),$$
a contradiction.
\end{proof}

Let $I\subset S$ be a monomial ideal, minimally generated by $m$ monomials.
As we already seen, if $I$ has linear quotients then $\slength(I)=m$.
It is natural to ask if the converse assertion is true, without the restrictive conditions
given in Theorem \ref{popo}(2).

We will see, later on (Example \ref{ex1}(4)), that, in general, 
the answer is negative. However, in some special cases, the answer is positive:

\begin{prop}\label{p41}
Let $I\subset S$ be a monomial ideal, minimally generated by $m$ monomials.
If $\slength(I)=m\leq 3$, then $I$ has linear quotients.
\end{prop}

\begin{proof}
We have $I=u_1K[Z_1]\oplus \cdots\oplus u_mK[Z_m]$, where $G(I)=\{u_1,\ldots,u_m\}$.
Without any loss of generality, we can assume that $Z_1=\{x_1,\ldots,x_n\}$ and $Z_i\subsetneq \{x_1,\ldots,x_n\}$, for $2\leq i\leq m$.

If $m=1$, then $I=(u_1)=u_1K[x_1,\ldots,x_n]$ and there is nothing to prove. 
If $m=2$, then the conclusion follows from Theorem \ref{popo}.

Assume $m=3$. In order to complete the proof, according to Theorem \ref{popo}, it is enough to show that
\begin{equation}\label{clem}
(u_1,u_2)=u_1K[Z_1]\oplus u_2K[Z_2]\text{ or }(u_1,u_3)=u_1K[Z_1]\oplus u_3K[Z_2].
\end{equation}
Assume, by contradiction, that \eqref{clem} is false. It follows
that there exists some monomials $w_2,w_3\in S$ such that
$$ w_2 \in (u_1,u_3)\setminus (u_1K[Z_1]\oplus u_3K[Z_3])\text{ and }w_3 \in (u_1,u_2)\setminus (u_1K[Z_1]\oplus u_3K[Z_2]).$$
Since $I=u_1K[Z_1]\oplus u_2K[Z_2]\oplus u_3K[Z_3]$ and $Z_1=\{x_1,\ldots,x_n\}$, it follows that
$$ w_2\in u_2K[Z_2],\;w_2\in (u_3),\;w_3\in u_3K[Z_3],\;w_3\in (u_2).$$
Since $w_2\notin u_3K[Z_3]$ and $u_3\mid w_2$ it follows that there exists some variable $x_{j_3}\in Z_3$, such that $x_{j_3}\mid w_2$.
Moreover, as $w_2\notin u_1K[Z_1]$ and $x_{j_3}u_3\mid w_2$, it follows that $x_{j_3}u_3\in u_2K[Z_2]$ and thus
$x_{j_3}u_3=\alpha u_2$, for some monomial $\alpha\in K[Z_2]$.

Similarly, there exists some variable $x_{j_2}\notin Z_2$, such that $x_{j_2}u_2 \in u_3K[Z_3]$ and thus 
$x_{j_2}u_2=\beta u_3$, for some monomial $\beta\in K[Z_2]$.

Since $x_{j_3}u_3=\alpha u_2$, $x_{j_2}u_2=\beta u_3$ and $u_2\neq u_3$, it follows that 
$$x_{j_2}u_2 = x_{j_3}u_3 \in u_2K[Z_2] \cap u_3K[Z_3],$$
a contradiction.
\end{proof}

\begin{prop}\label{p42}
Let $I\subset S=K[x_1,x_2]$ be a monomial ideal, minimally generated by $m$ monomials.
If $\slength(I)=m$, then $I$ has linear quotients.
\end{prop}

\begin{proof}
As in the proof of Theorem \ref{p28}, we may assume that 
$$G(I)=\{x_1^{a_1}x_2^{b_1},\;x_1^{a_2}x_2^{b_2},\ldots,x_1^{a_m}x_2^{b_m}\},$$
where $a_1>a_2>\cdots>a_m=0$ and $0=b_1<b_2<\ldots<b_m$. Moreover, from Theorem \ref{p28}, it follows
that $m=\min\{a_i+b_i-1\;:\;1\leq i\leq m\}$. Thus
$$G(I)=\{x_1^{a_1},x_1^{a_2}x_2,x_1^{a_3}x_2^2,\ldots,x_1^{a_i}x_2^{b_i},x_1^{a_i-1}x_2^{b_{i+1}},\ldots,x_1x_2^{b_{m-1}},x_2^{b_m}\}.$$
We consider the following order of $G(I)$:
\begin{align*}
& u_1=x_1^{a_i}x_2^{b_i},\;u_2=x_1^{a_{i-1}}x_2^{b_i-1},\ldots,u_{b_i}=x_1^{a_2}x_2,\;u_{b_i+1}=x_1^{a_1},\\
& u_{b_i+2}=x_1^{a_i-1}x_2^{b_{i+1}},\ldots, u_{a_i+b_i}=x_1x_2^{b_{m-1}},\;u_{a_i+b_i+1}=x_2^{b_m}.
\end{align*}
It follows that
\begin{align*}
& (u_1:u_2)=(x_2),\;(u_1,u_2):(u_3)=(x_2),\ldots,(u_1,\ldots,u_{b_i}):(u_{b_i+1})=(x_2),\\
& (u_1,\ldots,u_{b_i+1}):(u_{b_i+2})=(x_1),\ldots,(u_1,\ldots,u_{a_i+b_i}):(u_{a_i+b_i+1})=(x_1).
\end{align*}
Hence, $I$ has linear quotients, as required.
\end{proof}

\begin{exm}\label{ex1}\rm
\begin{enumerate}
\item[(1)] Let $\mathbf m=(x_1,x_2,x_3)\subset S=K[x_1,x_2,x_3]$. Since $\mathbf m$ has linear quotients, it follows that $\slength(\mathbf m)=3$.
For instance, $$\mathcal D\;:\;\mathbf m=x_1K[x_1,x_2,x_3]\oplus x_2K[x_2,x_3]\oplus x_3K[x_3],$$ is a Stanley decomposition of $\mathbf m$ of
length $3$. However, we have $$\sdepth(\mathcal D)=1<2=\sdepth(\mathbf m).$$ Note that, the shortest Stanley decomposition of $\mathbf m$ with its Stanley depth
equal to $2$, has the length $4$. Indeed, we have for instance:
$$\mathcal D': \mathbf m=x_1K[x_1,x_2]\oplus x_2K[x_2,x_3]\oplus x_3K[x_1,x_3]\oplus x_1x_2x_3K[x_1,x_2,x_3].$$
\item[(2)] Let $I=(x_1^2x_2,x_1x_2x_3,x_1^3x_3)\subset S=K[x_1,x_2,x_3]$. Obviously, $I$ is an ideal with linear quotients and, moreover
$$\mathcal D\;:\;I=x_1^2x_2K[x_1,x_2,x_3]\oplus x_1x_2x_3K[x_2,x_3]\oplus x_1^3x_3K[x_1,x_3],$$
is a Stanley decomposition of $I$ with $\slength(\mathcal D)=3=\slength(I)$ and $\sdepth(\mathcal D)=3=\sdepth(I)$.
\item[(3)] Let $I=(x_1x_2, x_1x_3, x_1x_4, x_2x_3x_4)\subset S=K[x_1,x_2,x_3,x_4]$. It is easy to see that $I$ has linear quotients, and 
\begin{align*}
\mathcal D\;:\;I= & x_1x_2K[x_1,x_2,x_3,x_4]\oplus x_1x_3 K[x_1,x_3,x_4]\\
& \oplus x_1x_4K[x_1,x_4]\oplus x_2x_3x_4 K[x_2,x_3,x_4],
\end{align*}
is the Stanley decomposition induced by the linear ordering on the generators of $I$. However, $I$ admits also the following decomposition:
\begin{align*}
\mathcal D'\;:\;I=& x_2x_3x_4K[x_1,x_2,x_3,x_4]\oplus x_1x_2 K[x_1,x_2,x_4] \\
& \oplus x_1x_3K[x_1,x_2,x_3]\oplus x_1x_4 K[x_1,x_3,x_4].
\end{align*}
Note that $\slength(\mathcal D)=\slength(\mathcal D')=4=\slength(I)$, but $\sdepth(\mathcal D)=2<\sdepth(\mathcal D')=\sdepth(I)=3$.
The decomposition $\mathcal D'$ is not induced by a linear order, since $(x_2x_3x_4:x_1x_2)=(x_3x_4)$.
\item[(4)] Let $I=(x_1^2x_2, x_1^2x_3, x_1^2x_4, x_2x_3x_4)\subset S=K[x_1,x_2,x_3,x_4]$. It is easy to see that $I$ does not have linear quotients.
On the other hand,
\begin{align*}
\mathcal D\;:\;I=& x_2x_3x_4K[x_1,x_2,x_3,x_4]\oplus x_1^2x_2 K[x_1,x_2,x_4] \\ 
& \oplus x_1^2x_3K[x_1,x_2,x_3]\oplus x_1^2x_4 K[x_1,x_3,x_4],
\end{align*}
is a Stanley decomposition of $I$. Hence $\slength(I)=4=$ the number of minimal monomial generators of $I$.
\end{enumerate}
\end{exm}

\begin{obs}\label{o44}\rm
Given a finitely generated multigraded $S$-module $M$, it would be interesting to find a Stanley decomposition of minimal length
which realize the Stanley depth of $M$. However, such a decomposition is not necessarily of minimal length. For instance, according
to Example \ref{ex1}(1), although $\slength(\mathbf m)=3$, where $\mathbf m=(x_1,x_2,x_3)\subset S=K[x_1,x_2,x_3]$,
a minimal Stanley decomposition $\mathcal D$ of $\mathbf m$ with $\sdepth(\mathcal D)=\sdepth(\mathbf m)=2$ must have the length at least $4$.
\end{obs}

\begin{teor}\label{t5}
Let $I\subset S$ be a squarefree complete intersection monomial ideal, minimally generated by the monomials $u_1,\ldots,u_m\in S$.
Let $d_i=\deg(u_i)$, $1\leq i\leq m$, and assume that there exists some $k\geq 0$ such that 
$$1=d_1=d_2=\cdots=d_k<d_{k+1}\leq d_{k+2} \leq \cdots \leq d_m.$$
Let $n'=d_{k+1}+d_{k+2}+\cdots+d_m$. Then $\slength(I)\leq k+\binom{n'+1}{\left\lfloor \frac{m-k}{2} \right\rfloor}$.

In particular, if $k=0$ and $n'=n$, then $\slength(I)\leq \binom{n+1}{\left\lfloor \frac{m}{2} \right\rfloor}$.
\end{teor}

\begin{proof}
Using Lemma \ref{lem2}, we can assume that $k=0$ and then, using Proposition \ref{p5}, we can assume also $n'=n$.
That is $$2\leq d_1\leq d_2\leq \cdots \leq d_m\text{ and }d_1+d_2+\cdots+d_m=n.$$
Note that, any squarefree monomial $u\in S$ with $\deg(u)\geq n-m+1$, belongs to $I$.
Since $I$ is a monomial complete intersection, according to \cite[Theorem 2.4]{shen}, we have 
$\sdepth(I)=n-\left\lfloor \frac{m}{2} \right\rfloor$. Using \cite[Theorem 3.5]{rin} and Proposition \ref{propi}, we can find a partition
$$\mathcal P\;:\;\Delta(I)=\bigcup_{i=1}^r [F_i,G_i]$$
of $\Delta(I)$, such that if $|F_i|<n-\left\lfloor \frac{m}{2} \right\rfloor$, then $|G_i|=n-\left\lfloor \frac{m}{2} \right\rfloor$, and,
if $|F_i|=n-\left\lfloor \frac{m}{2} \right\rfloor$, then $G_i=F_i$. Let $J$ be the squarefree Veronese ideal, generated by 
all the squarefree monomials of degree $n-\left\lfloor \frac{m}{2} \right\rfloor+1$. Note that $J\subset I$, since 
$n-m+1\geq n-\left\lfloor \frac{m}{2} \right\rfloor+1$ and all the squarefree monomials of degree $\geq n-m+1$ belongs to $I$.

From the above considerations, it follows that
$$\Delta(I/J)=\bigcup_{|F_i|\leq n-\left\lfloor \frac{m}{2} \right\rfloor} [F_i,G_i].$$
Since, in the above decomposition of $\Delta(J/I)$, all $G_i$'s have cardinality $n-\left\lfloor \frac{m}{2} \right\rfloor$, it follows that
\begin{equation}\label{ecus1}
\slength(I/J)\leq \binom{n}{n-\left\lfloor \frac{m}{2} \right\rfloor}=\binom{n}{\left\lfloor \frac{m}{2} \right\rfloor}.
\end{equation}
On the other hand, since $J$ has obviously linear quotients, with respect to the lexicographic order on the monomial generators of $J$,
it follows hat 
\begin{equation}\label{ecus2}
\slength(J)=\binom{n}{n-\left\lfloor \frac{m}{2}+1 \right\rfloor}=\binom{n}{\left\lfloor \frac{m}{2} \right\rfloor-1}.
\end{equation}
From \eqref{ecus1}, \eqref{ecus2} and Proposition \ref{p1}, applied to $0\to J\to I\to I/J\to 0$, it follows that
$$\slength(I)\leq \binom{n}{\left\lfloor \frac{m}{2} \right\rfloor} + \binom{n}{\left\lfloor \frac{m}{2} \right\rfloor-1}
=\binom{n+1}{\left\lfloor \frac{m}{2} \right\rfloor},$$ as required.
\end{proof}

\begin{obs}\rm
Note that, if $I\subset S$ is a squarefree complete intersection, minimally generated by $3$ monomials, then, 
from Theorem \ref{t5}, it follows that $\slength(I)\leq n+1$. In this case, the same conclusion follows also from 
Proposition \ref{p-trei} or Proposition \ref{3sq}.
\end{obs}

\begin{exm}\label{ext5}\rm
(1) Let $I=(x_1,\;x_2,\;x_3x_4,\;x_5x_6,x_7x_8x_9x_{10})\subset S=K[x_1,\ldots,x_{10}]$. We have $m=5$ and 
$1=d_1=d_2<d_3=d_4=2<d_5=4$. Hence $k=2$ and $n'=d_3+d_4+d_5=8$. According to Theorem \ref{t5} we have
$$\slength(I)\leq k + \binom{n'+1}{\left\lfloor \frac{m-k}{2} \right\rfloor} = 2 + \binom{9}{1} = 11.$$
(2) Let $I=(x_1x_2,x_3x_4,x_5x_6)\subset S=K[x_1,\ldots,x_6]$. We have $m=3$ and $d_1=d_2=d_3=2$.
According to Theorem \ref{t5} or Proposition \ref{p-trei} we have $\slength(I) \leq  7$.
\end{exm}

\subsection*{Data availability}

Data sharing not applicable to this article as no datasets were generated or analyzed during the current study.

\subsection*{Conflict of interest}

The author has no relevant financial or non-financial interests to disclose.

{}

\end{document}